\documentclass[10pt]{article}  

\usepackage{amsmath, amsthm, amsfonts, amssymb}
\usepackage[colorlinks=true,linkcolor=red!50!black,citecolor=green!50!black,urlcolor=blue!50!black]{hyperref}
\usepackage{indentfirst} 
\usepackage{marvosym}
\usepackage{enumitem}
\usepackage[a4paper]{geometry}
\usepackage{xcolor}
\usepackage{microtype}

\newtheorem{lemma}{Lemma}[section]
\newtheorem{theorem}[lemma]{Theorem}
\newtheorem{proposition}[lemma]{Proposition}

\renewenvironment{proof}[1][\proofname]{{\noindent\bf #1. }}{\qed}
\newtheorem{theoremletters}{Theorem}

\newtheorem{conjectureletters}[theoremletters]{Conjecture}

\theoremstyle{definition}
\newtheorem{example}[lemma]{Example}

\newcommand{\abs}[1]{\ensuremath{|#1|}}
\newcommand{\op}{\operatorname}
\newcommand{\ce}[2]{\pmb{\op{C}}_{#1}(#2)}

\newcommand{\ze}[1]{\pmb{\op{Z}}(#1)}

\newcommand{\fitdos}[1]{\pmb{\op{F}}_2(#1)}
\newcommand{\rad}[2]{\pmb{\op{O}}_{#1}(#2)}

\newcommand{\hall}[2]{\op{Hall}_{#1}\left(#2\right)}

\begin{document}

\title{\bf Finite groups in which almost all class sizes are prime numbers}

\date{\textit{\small Dedicated to Professors P. Longobardi and M. Maj on the occasion of their retirement.}}

\author{\sc Carmine Monetta\thanks{Dipartimento di Matematica, Università di Salerno, Via Giovanni Paolo II, 132-84084 Fisciano, Italy. \newline \Letter: \texttt{cmonetta@unisa.it} \newline ORCID: 0000-0002-5426-1260}\; and Víctor Sotomayor\thanks{Departamento de Álgebra, Universidad de Granada, Avenida de Fuentenueva s/n, 18071 Granada, Spain. \newline
\Letter: \texttt{vsotomayor@ugr.es} \newline 
ORCID: 0000-0001-8649-5742 \newline \rule{6cm}{0.1mm}\newline
The research of the first author is funded by the European Union - Next Generation EU, Missione 4 Componente 1 CUP B53D23009410006, PRIN 2022 - 2022PSTWLB - Group Theory and Applications. This research has been carried out during a visit of the second author to the Dipartimento di Matematica (DipMat) of Università degli Studi di Salerno, with the financial support of the grant CIBEST/2024/184 from Generalitat Valenciana (Spain) and of the National Group for Algebraic and Geometric Structures, and their Applications (GNSAGA – INdAM). He wishes to acknowledge all the members of DipMat for the hospitality. \newline
}}

\maketitle

\begin{abstract}
\noindent In this paper we study arithmetical and structural features of a finite group that possesses exactly two conjugacy class sizes that are composite numbers.

\medskip


\smallskip

\end{abstract}


\section{Introduction}

All groups considered hereafter are assumed to be finite. Let $x^G$ be the conjugacy class of an element $x$ of a group $G$, and let $cs(G)$ denote the set of all conjugacy class sizes of $G$. The relationship that exists between the algebraic structure of $G$ and the arithmetical properties of the set $cs(G)$ has been a target of increasing interest
within the theory of finite groups. The general purpose of these investigations is to provide answers to the following two questions: which algebraic properties of $G$ are encoded in arithmetical patterns of $cs(G)$, and which sets of positive integers may occur as $cs(G)$ for some group $G$? We refer the interested reader to the survey \cite{CCsurvey} for a detailed discussion on this subject.

The aim of this paper is to aid the study of the aforementioned questions by analysing groups with few conjugacy class sizes that are composite numbers (that is, positive integers that have at least one divisor different from 1 and themselves). The extreme situation where, for a group $G$, no element of $cs(G)$ is a composite number was studied by D. Chillag and M. Herzog in \cite{CH}: up to abelian direct factors, $G$ is either a $p$-group of nilpotency class at most $2$ and $cs(G)=\{1,p\}$, for some prime $p$, or $G/\ze{G}$ is a Frobenius group of order $pq$ and $cs(G)=\{1,p,q\}$, for some primes $q\neq p$. Three decades later, Q. Jiang \emph{et al.} focused on groups $G$ such that $cs(G)$ has exactly one composite class size, and they also obtained a classification of the structure of $G$ (\emph{cf.} \cite{JSZ}). Inspired by this research, our goal in this paper is to study algebraic properties of $G$ and arithmetical features of $cs(G)$ when this last set possesses exactly two composite numbers. The following is our main result.

\begin{theoremletters}
\label{theoA}
Let $G$ be a group with exactly two conjugacy class sizes that are composite numbers. Then $G$ has at most three prime class sizes. Moreover, if $cs(G)=\{1,p_1,p_2,p_3,m,n\}$, where $m$ and $n$ are the unique composite numbers in $cs(G)$, then up to interchanging the primes $p_i$ it holds $m=p_1p_2$ and $n=p_1p_3$; further, up to abelian direct factors, $G=P\times H$ where $P$ is a $p_1$-group of nilpotency class at most $2$ with $cs(P)=\{1,p_1\}$, and $H/\ze{H}$ is a Frobenius group of order $p_2p_3$ with $cs(H)=\{1,p_2,p_3\}$.
\end{theoremletters}

We point out that the information obtained on the $p$-structure of a group $G$ from the class sizes of its $p'$-elements, for some prime $p$, has proven to be a fundamental tool for addressing the proof of Theorem~\ref{theoA}. This fact illustrates once again that the study of certain relevant subsets of $cs(G)$ provides not only more general theorems, but also results concerning local information of a group that may be essential in other settings to obtain global information on the structure of the group.

As a consequence of Theorem \ref{theoA}, any group $G$ with exactly two composite class sizes satisfies $3\leq |cs(G)|\leq 6$. Moreover, observe that $G$ is soluble in both extreme cases, when $|cs(G)|$ is either $3$ or $6$: this  follows from the main result of \cite{Ito} and Theorem~\ref{theoA}, respectively. The authors have not been able to prove the solubility of $G$ in the remaining cases, \emph{i.e.} when $cs(G)$ contains exactly two composite numbers and its size is either $4$ or $5$, although they believe that this property holds in full generality:

\medskip

\begin{conjectureletters}
\label{conjecture}
If $G$ has exactly two composite conjugacy class sizes, then $G$ is soluble.
\end{conjectureletters}

In Proposition \ref{almostsimple} we shed some light on this phenomenon by proving that almost simple groups cannot possess exactly two composite class sizes; as the reader may expect, this appeals to the classification of finite simple groups. We will also show that the solubility of the group follows when one of these composite numbers is a prime power. In that situation we can provide additional information about $G$ and the arithmetical structure of $cs(G)$, which is summarised in the next theorem. Recall that groups with $|cs(G)|=3$ have been completely characterised by S. Dolfi and E. Jabara in \cite{DJ}, so we may avoid that case in the theorem below.

\begin{theoremletters}
\label{theoC}
Let $G$ be a group with $|cs(G)|\geq 4$. If $G$ has exactly two conjugacy class sizes that are composite numbers, and one of them is a prime power, then $G$ is soluble, and either $G=\fitdos{G}$ or every non-identity element of $G/\fitdos{G}$ has prime order. Moreover, if $p^a$ and $n$ are the two composite class sizes, for some prime $p$ and some integer $a\geq 2$, then $p$ divides $n$, and one of the following cases holds:
\begin{enumerate}[label=\emph{(\alph*)}]
\setlength{\itemsep}{0mm}
\item $cs(G)=\{1,p,p^a,n\}$. Further, if $n=p^b$ with $b\notin \{1,a\}$, then $G$ is a $p$-group up to abelian direct factors;
\item $cs(G)=\{1,q,p^a,qp^b\}$ for a prime $q\neq p$ and an integer $1\leq b\leq a$, so $n=qp^b$, and $G$ is a $\{p,q\}$-group up to abelian direct factors;
\item $cs(G)=\{1,p,q,p^a,n\}$ with $n$ a $\{p,q\}$-number, for a prime $q\neq p$, and $G$ is a $\{p,q\}$-group up to abelian direct factors.
\end{enumerate}
\end{theoremletters}

We remark that the description of soluble groups whose non-identity elements have prime orders is well-known (\emph{cf.} \cite[Proposition 2.8]{DJ}). In addition, the algebraic structure of $G$ in the previous statement (b) can be further classified by \cite[Main theorem]{SJ} and \cite[Theorem 1]{CCnil}. Regarding the concrete description of $n$ in (b), it has been provided by virtue of \cite[Theorem A]{BF}; in fact, we strongly suspect that the structure of $n$ in statement (c) is also of the form $qp^b$ for some integer $1\leq b\leq a$. Contrariwise, it may happen in (a) that $n$, and consequently the order of $G/\ze{G}$, is divisible by more than two primes, which contrasts with cases (b) and (c). Historically, there have been many unsolved problems on determining the structure of groups with four class sizes, especially when no pair of them are coprime (see \cite{BFsurvey} and \cite{CCsurvey}). Thus, it seems challenging to provide a detailed description about the structure of $G$ in (a). At the end of the paper we include a list of examples that illustrate all these features.

To close with, we remark that in this research area there exist several analogies between theorems proved for conjugacy class sizes of a group and the corresponding problems for the irreducible character degrees. As it might be expected, the analysis carried out in this paper has a ``character degree version'', addressed by M. Ghaﬀarzadeh and M. Ghasemi in \cite{GG}.


\section{Preliminaries}

In the following, we write $\pi(n)$ for the set of prime divisors of a positive integer $n$. In particular, $\pi(G)$ and $\pi(x^G)$ are the sets of prime divisors of $\abs{G}$ and $|x^G|=|G:\ce{G}{x}|$, respectively. If $p$ is a prime, then we denote by $n_p$ the largest $p$-power that divides $n$, by $p'$ the set of primes different from $p$, and by $n_{p'}=n/n_p$ which is the product of $n_q$ for each prime $q\in p'$. The greatest common divisor of two positive integers $a$ and $b$ will be written as $(a,b)$. Recall that each element $g\in G$ can be decomposed as product of pairwise commuting elements of prime power order, say $g_{q_1},...,g_{q_s}$, for some integer $s\geq 1$ and certain primes $\{q_1,...,q_s\}$, and each $q_i$-element $g_{q_i}$ is called the $q_i$-part of $g$; we call this the \emph{primary decomposition} of $g$. The remaining notation and terminology used are standard in the framework of group theory.

Let us start with some basic and well-known facts.

\begin{lemma}
\label{basic}
Let $G$ be a group and $N$ be a normal subgroup of $G$. 
\vspace*{-2mm}
\begin{itemize}
\setlength{\itemsep}{0mm}
\item[\emph{(a)}] If $x\in G$, then $|(xN)^{G/N}|$ divides $|x^G|$.
\item[\emph{(b)}] If $(|x^G|,|y^G|)=1$ for certain $x,y\in G$, then $G=\ce{G}{x}\ce{G}{y}$, so $(xy)^G=x^Gy^G$ and $|(xy)^G|$ divides $|x^G|\cdot|y^G|$.
\item[\emph{(c)}] If $x,y\in G$ have coprime orders and they commute, then $\ce{G}{xy}=\ce{G}{x}\cap\ce{G}{y}$. In particular, both $|x^G|$ and $|y^G|$ divide $|(xy)^G|$.
\item[\emph{(d)}] A prime $p$ does not divide any element of $cs(G)$ if and only if $G=P\times \rad{p'}{G}$, where $P$ is an abelian Sylow $p$-subgroup of $G$.
\item[\emph{(e)}] If $xN\in G/N$ is a $\pi$-element, for some set of primes $\pi$, then $xN=yN$ for some $\pi$-element $y\in G$.
\end{itemize}
\end{lemma}

\begin{proof}
Statements (a) and (b) follows from \cite[Lemma 1.1]{CH} and \cite[Lemma J (b)]{BK}, respectively, whilst (c) and (d) are \cite[Lemma 2.1 (ii) and Lemma 2.2]{DJ}. The last assertion follows immediately by considering the primary decomposition of $x$.
\end{proof}

\medskip

In our discussion we will frequently deal with class sizes that are coprime, and the next results will be very useful.

\begin{lemma}\emph{(\cite[Lemma 1]{CCcoprime})}
\label{lemma_CC}
Let $G$ be a group and $1<a<b_1<b_2$ be pairwise coprime elements of $cs(G)$. Then there exist some $i\in \{1,2\}$ and some $c\in cs(G)$ such that $c>b_i$ and $c$ divides $ab_i$.
\end{lemma}

\smallskip

\begin{theorem}
\label{disconnected}
Let $G$ be a group and $\pi$ be a set of primes. Suppose that each element of $cs(G)$ is either a $\pi$-number or a $\pi'$-number, and that both cases occur. Then, up to abelian direct factors (and up to interchanging $\pi$ and $\pi'$), $G=HL$ with $H\in \hall{\pi}{G}$, $L\in\hall{\pi'}{G}$, $L\unlhd G$, both $H$ and $L$ abelian, and $G/\ze{G}$ a Frobenius group. In particular, $cs(G)=\{1,|L|, |H/\ze{G}|\}$.
\end{theorem}

\begin{proof}
This is an immediate consequence of \cite[Theorem]{K}.
\end{proof}

\medskip

The following are two key ingredients for proving Theorem~\ref{theoA}.

\begin{lemma}\emph{(\cite[Lemma 4]{BK})}
\label{lemma_kazarin}
Let $G$ be a group and $t$ be a prime. Suppose that the $t$-elements $x,y \in G \setminus Z(G)$ are such that $|x^G|$ and $|y^G|$ are powers of distinct primes, and $|(xy)^G|$ is also a power of a prime. Then $\langle x,y \rangle \leq \rad{t}{G}$, and $|(xy)^G|= \max\{|x^G|,|y^G|\}$ is a power of $t$, so a Sylow $t$-subgroup of $G$ is non-abelian.
\end{lemma}

\smallskip

\begin{theorem}\emph{(\cite[Lemma 2.4]{JS})}
\label{theo-2-reg}
Let $G$ be a group and $p$ be a prime. If every prime power order $p'$-element has class size either $1$ or $m$, for some fixed positive integer $m$, then $m=p^aq^b$ for some prime $q\neq p$ and some integers $a,b\geq 0$. In particular, up to abelian direct factors, $G$ is a $\{p,q\}$-group.
\end{theorem}

Finally, we include a result due to N. Itô concerning \emph{minimal} centralisers.

\begin{lemma}
\label{ito}
Let $G$ be a group. If $X$ is a centraliser that does not properly contain other centraliser, and there exists an $r$-element $g\in G$ whose centraliser is $X$, then $X=R \times A$ where $R$ is a Sylow $r$-subgroup of $X$ and $A$ is an abelian $r'$-group.
\end{lemma}

\begin{proof}
Take an $r'$-element $h\in X=\ce{G}{g}$. Since $gh=hg$ and they have coprime orders, then by Lemma~\ref{basic} (c) it holds $\ce{G}{gh}=\ce{G}{g}\cap\ce{G}{h}\leqslant\ce{G}{g}$. But $\ce{G}{g}$ does not properly contain other centraliser, so $\ce{G}{gh}=\ce{G}{g}\cap\ce{G}{h}=\ce{G}{g}$. Therefore $$|h^{\ce{G}{g}}|=|\ce{G}{g}:\ce{\ce{G}{g}}{h}|=|\ce{G}{g}:\ce{G}{g}\cap\ce{G}{h}|=1.$$ We conclude that every $r'$-element of $\ce{G}{g}$ is central, so the claim follows.
\end{proof}

\section{Proofs of main results}

We start this section by proving our main theorem.

\smallskip

\begin{proof}[Proof of Theorem~\ref{theoA}]
Observe that $cs(G\times A)=cs(G)$ for any abelian group $A$, so we may assume from now on that $G$ has no abelian direct factor. We first show that $G$ has at most three prime class sizes. Arguing by contradiction, we suppose that there exist primes $p_1<p_2<p_3<p_4$ in $cs(G)$. Lemma~\ref{lemma_CC} applied to $p_1<p_2<p_3$ provides a class size $c\in cs(G)$ that divides $p_1p_i$, for some $i\in \{2,3\}$, with $c>p_i$. Thus $c=p_1p_i$ necessarily. The same argument applied to $p_2<p_3<p_4$ yields a class size $d=p_2p_j$ for some $j\in \{3,4\}$. Consequently, by assumptions, the unique two composite elements of $cs(G)$ are $c$ and $d$, so $cs(G)=\{1,p_1p_i, p_2p_j, p_1,p_2,p_3,p_4,...,p_s\}$ where $p_s$ is a prime number for all $s\geq 4$. Now, on the one hand, if $(c,d)=1$ (\emph{i.e.} $i=3$ and $j=4$), then we can apply Theorem~\ref{disconnected} with $\pi=\{p_1,p_3\}$ and we get the contradiction $|cs(G)|=3$. On the other hand, if $(c,d)>1$ (\emph{i.e.} the greatest common divisor between $c$ and $d$ is either $p_2$ or $p_3$), then we also get a contradiction due to Theorem~\ref{disconnected} with $\pi=\{p_1,p_2,p_j\}$. 

Let us suppose now that $cs(G)=\{1,p_1,p_2,p_3,m,n\}$, where $m$ and $n$ are the unique composite numbers in $cs(G)$. The remaining part of the proof aims to show that $m=p_ip_j$ and $n=p_ip_k$, where $\{i,j,k\}=\{1,2,3\}$. Note that, in that situation, it holds $cs(G)=\{1,p_i\}\times\{1,p_j,p_k\}$, so \cite[Theorem 1]{CCnil} ensures that $G=P\times H$ with $cs(P)=\{1,p_i\}$ and $cs(H)=\{1,p_j,p_k\}$. In fact, $P$ is a $p_i$-group with nilpotency class at most $2$ by \cite[Corollary 2.3]{CH}, and $H/\ze{H}$ is a Frobenius group of order $p_jp_k$ due to Theorem~\ref{disconnected}, as wanted.

Suppose that $|x^G|=p_1$, $|y^G|=p_2$ and $|z^G|=p_3$ for certain elements $x,y,z\in G$. In virtue of Lemma~\ref{basic} (c) and the primary decomposition of $x$, $y$ and $z$, we may suppose that $x$ is a $t_1$-element, $y$ is a $t_2$-element and $z$ is a $t_3$-element, for certain primes $t_1$, $t_2$ and $t_3$. We first claim that the thesis about $m$ and $n$ is true whenever the primes $t_i$ are pairwise distinct. In fact, since $t_1$ is different from one of $\{p_2,p_3\}$, say $t_1\neq p_2$, up to conjugation we may affirm that $xy=yx$. Thus by Lemma~\ref{basic} (b-c) we get $|(xy)^G|=p_1p_2=m$. Arguing similarly, since $t_3$ is different from one of $\{p_1,p_2\}$, then we obtain a conjugacy class with size either $p_1p_3=n$ or $p_2p_3=n$, as desired.

Hence we may suppose, without loss of generality, that $t_1=t_2=t$ for some prime $t$. Observe that if $t\neq t_3$, since $t_3$ is different from one of $\{p_1,p_2\}$, then we obtain by Lemma~\ref{basic} (b-c) that either $|(xz)^G|=p_1p_3$ or $|(yz)^G|=p_2p_3$. In particular, if $t_3$ is different from both $p_1$ and $p_2$, then we are done. Therefore, we have the following two cases to discuss.

\smallskip

\noindent \textbf{\underline{Case I:}} $x$ and $y$ are $t$-elements, and $z$ is a $p_i$-element with $i\in \{1,2\}$ and $t\neq p_i$.

\smallskip

Without loss of generality we may affirm that $z$ is a $p_2$-element. Then, up to conjugation, $zx=xz$ and so by Lemma~\ref{basic} (b-c) it follows $|(xz)^G|=p_1p_3=m$. If moreover $t\neq p_3$, then we can obtain the class size $|(yz)^G|=p_2p_3=n$, as wanted. So we may suppose that both $x$ and $y$ are $p_3$-elements. Observe that $|(xy)^G|$ divides $|x^G|\cdot|y^G|=p_1p_2$ by Lemma~\ref{basic} (b), so $|(xy)^G|$ is either a prime number or equal to $p_1p_2=n$. However, the first case cannot occur, otherwise $|(xy)^G|=\op{max}\{p_1,p_2\}$ would be a power of $p_3$. Thus $n=p_1p_2$ and we are done.

\smallskip

\noindent \textbf{\underline{Case II:}} $x$, $y$ and $z$ are $t$-elements, for some prime $t$.

\smallskip

Note that each element in $\{|(xy)^G|, |(yz)^G|, |(xz)^G|\}$ is either a prime number or a product of two primes by Lemma~\ref{basic} (b). Since we have exactly two composite class sizes by assumptions, then certainly all of them cannot be composite numbers. Further, they cannot all be prime numbers, because in that case Lemma~\ref{lemma_kazarin} yields that $|(xy)^G|=\op{max}\{p_1,p_2\}$, $|(yz)^G|=\op{max}\{p_2,p_3\}$ and $|(xz)^G|=\op{max}\{p_1,p_3\}$ are powers of $t$, which is impossible. Additionally, if only two of $\{|(xy)^G|, |(yz)^G|, |(xz)^G|\}$ are composite numbers, then we are done. Consequently, it remains to study the situation where two of $\{|(xy)^G|, |(yz)^G|, |(xz)^G|\}$ are prime numbers and the other is a composite number. 

By symmetry, let us suppose that $|(xy)^G|=p_1p_2=m$ and that both $|(yz)^G|$ and $|(xz)^G|$ are prime numbers. Lemma~\ref{lemma_kazarin} leads to $|(yz)^G|=\op{max}\{p_2,p_3\}=t=\op{max}\{p_1,p_3\}=|(xz)^G|$, so clearly $t=p_3$ and in particular $x$, $y$ and $z$ are $p_3$-elements. 

Now we show that, if $|h^G|$ is prime for some $h\in G\smallsetminus \ze{G}$, then we may suppose that the $q$-part $h_q\in\ze{G}$ for every $q\neq p_3$. Otherwise $|h_q^G|=|h^G|$ is prime by Lemma~\ref{basic} (c). If $|h_q^G|\in\{p_1,p_2\}$, then $hz=zh$ up to conjugation, so we can produce a conjugacy class of size either $p_1p_3=n$ or $p_2p_3=n$. On the other hand, if $|h_q^G|=p_3$, since $q$ is different from either $p_1$ or $p_2$, it is easy to see that we then obtain either $p_1p_3=n$ or $p_2p_3=n$ as a class size of $G$, as desired.

By the previous paragraph, we may affirm that there is no ${p_{3}}'$-element in $G$ with prime class size. Let us suppose that there is a ${p_{3}}'$-element $g\in G$ with $|g^G|=m=p_1p_2$. Hence, up to conjugation, $zg=gz$ and so by Lemma~\ref{basic} (b-c) we get $|(gz)^G|=p_1p_2p_3=n$. In particular $cs(G)=\{1,p_1,p_2,p_3,p_1p_2, p_1p_2p_3\}$, and by \cite[Theorem 1]{CH} it holds that $G$ is (super)soluble. If there exists a ${p_{3}}'$-element with class size $p_1p_2p_3$, then the class size of each ${p_{3}}'$-element of $G$ belongs to $\{1,p_1p_2,p_1p_2p_3\}$. An application of \cite[Theorem A]{ABFK} yields that $m=p_1p_2$ is a prime power, a contradiction. Therefore the ${p_{3}}'$-elements of $G$ all have class size either $1$ or $m=p_1p_2$. By Theorem~\ref{theo-2-reg} we obtain that $m$ is either a prime power or divisible by $p_3$, which is impossible.

We conclude that there is no ${p_{3}}'$-element in $G$ with class size either prime or equal to $m=p_1p_2$, \emph{i.e.} every ${p_{3}}'$-element of $G$ has class size equal to $1$ or $n$. In virtue again of Theorem~\ref{theo-2-reg}, we deduce $|\pi(G/\ze{G})|\leq 2$, which is not possible since $\{p_1,p_2,p_3\}\subseteq cs(G)$. The proof is therefore finished.
\end{proof}

\medskip

In the following result we provide some evidence for the veracity of Conjecture \ref{conjecture}. For its proof, we have made use of \cite[Theorem A]{many}, which depends on the classification of finite simple groups.

\begin{proposition}
\label{almostsimple}
If $G$ has exactly two composite conjugacy class sizes, then $G$ is not an almost simple group.
\end{proposition}

\begin{proof}
Arguing by contradiction, let us suppose that $G$ is almost simple. If $|cs(G)|\geq 5$, then $cs(G)$ contains two non-trivial coprime numbers, which contradicts \cite[Theorem A]{many}. It follows that $G$ has at most three distinct class sizes greater than 1, and as $G$ has trivial centre, then we deduce that $G$ is isomorphic to $PSL(2,2^a)$, for certain integer $a\geq 2$, by \cite[Corollary A]{CCcoprime}. In particular $cs(G)=\{1, 2^{2a}-1, 2^a(2^a-1), 2^a(2^a+1)\}$, so $G$ possesses three composite class sizes, a contradiction.
\end{proof}

Next we focus on the proof of Theorem~\ref{theoC}.

\smallskip

\begin{proof}[Proof of Theorem~\ref{theoC}]
We first show that $G$ is soluble. Since $G$ has exactly two class sizes that are composite numbers and one of them is a prime power, then the greatest common divisor between each pair of elements of $cs(G)$ is either $1$ or a non-trivial prime power. Thus \cite[Proposition 3.2]{CCone} ensures that $G$ is soluble.

Next we show that either $G=\fitdos{G}$ or every non-identity element of $G/\fitdos{G}$ has prime order. By assumptions, $p^a$ (with $a>1$) and $n$ are the unique composite numbers of $cs(G)$.
Note that all elements with prime power class size lie in $\fitdos{G}$ by \cite[Theorem 1]{CCimplications}.
In particular, if $n$ is also a prime power, then $G=\fitdos{G}$, as wanted. So we may suppose that $n$ is not a prime power, and then the elements lying in $G\smallsetminus\fitdos{G}$ all have class size equal to $n$.
In virtue of the main result of \cite{Isaacs} we get that either every non-trivial element of $G/\fitdos{G}$ has prime order, or $G$ has abelian Hall $\pi$-subgroups for $\pi=\pi(G/\fitdos{G})$. In the latter case, observe that for any non-trivial element  $x\fitdos{G}\in G/\fitdos{G}$ we may assume that $x$ is a $\pi$-element by Lemma~\ref{basic} (e), and $|x^G|=n$ necessarily because $x\notin \fitdos{G}$. Since the Hall $\pi$-subgroups of $G$ are abelian, and $x$ lies in some of them, then $n$ would be a $\pi'$-number. Hence $|\pi|=1$, because $\pi\subseteq \pi(G/\ze{G})$, and in particular every element of $cs(G)$ is either a $\pi$-number or a $\pi'$-number, both occurring. Now Theorem~\ref{disconnected} leads to a contradiction, so $G=\fitdos{G}$.

Observe that $|cs(G)|\leq 5$ by Theorem~\ref{theoA}. Moreover, note that $p$ must necessarily divide $n$: otherwise $cs(G)\smallsetminus\{1\}$ can be partitioned into two disjoint non-trivial subsets (being one of them the $p$-powers), but this contradicts the assumption $|cs(G)|\geq 4$ by Theorem~\ref{disconnected}. 

If $|cs(G)|=4$, then we have only two possibilities: either $cs(G)=\{1,q,p^a,n\}$ with $q\neq p$, or $cs(G)=\{1,p,p^a,n\}$. In the former case, applying \cite[Theorem A]{BF}, we get that $n=qp^b$ for some positive integer $b\leq a$. In particular, up to abelian direct factors, $G$ is a $\{p,q\}$-group by Lemma~\ref{basic} (d), which proves (b). In the latter case, if additionally $n=p^b$ for some positive integer $b\notin\{1,a\}$, then by Lemma~\ref{basic} (d) we get that $G=P\times A$ where $P$ is a Sylow $p$-subgroup and $A$ is abelian, and this yields statement (a).

For the rest of the proof let us suppose that $|cs(G)|= 5$, aiming to prove that statement (c) holds. Recall that $p$ divides $n$ and $\{1,p^a, n\}\subseteq cs(G)$, so that $cs(G)=\{1,q,r,p^a,n\}$ for some primes $q\neq r$. If $p\notin \{q,r\}$, then we can apply Lemma~\ref{lemma_CC} to $\{q,r,p^a\}$ and we deduce that $|\pi(n)|=2$, so $\pi(n)$ contains $p$ and either $q$ or $r$. Now Theorem~\ref{disconnected} applied with $\pi=\pi(n)$ leads to the contradiction $|cs(G)|= 3$. Hence $cs(G)=\{1,q,p,p^a,n\}$, for some prime $q\neq p$. A similar argument also shows that $q$ divides $n$, so $\{p,q\}\subseteq \pi(n)$. Note that $\pi(n)=\{p,q\}$ implies by Lemma~\ref{basic} (d) that each Sylow $r$-subgroup of $G$, for every prime $r\notin \{p,q\}$, is an abelian direct factor of $G$, so we are done. Hence we may suppose that $|\pi(n)|\geq 3$ and then we aim to get a contradiction.

Take $x,y\in G$ such that $|x^G|=q$ and $|y^G|=p$. We may certainly assume that $x$ is a $t_1$-element and $y$ is a $t_2$-element, for certain primes $t_1$ and $t_2$. We claim that $t_1=t_2$. Arguing by contradiction, if in addition $q\neq t_2$, then up to conjugation we may affirm that $xy=yx$ and Lemma~\ref{basic} (b-c) yields $|(xy)^G|=pq=n$, which is not possible. So we may affirm $t_2=q$, and similarly that $t_1=p$. Take a prime $r\in \pi(n)\smallsetminus \{p,q\}$. Since $|x^G|$ and $|y^G|$ are not divisible by $r$, then there exists a Sylow $r$-subgroup $R$ of $G$ such that $R\leqslant \ce{G}{x}\cap\ce{G}{y}$. Pick some $g\in R\smallsetminus \ze{G}$, which should exist due to the fact $r\in \pi(n)$. Thus $gx=xg$ and $gy=yg$ with $(o(x), o(g))=1=(o(y),o(g))$. Now, if $|g^G|\in\{q,p,p^a\}$ then it follows that $n\in\{qp, qp^a\}$ in virtue of Lemma~\ref{basic} (b-c), a contradiction. Thus $|g^G|=n$, and by Lemma~\ref{ito} one can see that $\ce{G}{g}=R_0\times A$ where $R_0$ is a Sylow $r$-subgroup of $\ce{G}{g}$ and $A$ is an abelian $r'$-group. As $x,y\in A$, then $xy=yx$ and by Lemma~\ref{basic} (b-c) it holds $|(xy)^G|=pq=n$, which is impossible.

From the above paragraph we deduce that $t_1=t_2=t$ for some prime $t$. Observe that $|(xy)^G|$ divides $pq$ by Lemma~\ref{basic} (b), and since we are supposing that $|\pi(n)|\geq 3$, then $|(xy)^G|$ must be prime. Now Lemma~\ref{lemma_kazarin} asserts that $|(xy)^G|=\op{max}\{p,q\}$ is a power of $t$, so $t\in\{p,q\}$. 

Pick an $s$-element $g\in G\smallsetminus \ze{G}$, for some prime $s\neq t$. We claim that either $|g^G|=n$, $|g^G|=q$ with $g$ a $p$-element (\emph{i.e.} $s=p$) and $t=q$, or $|g^G|\in\{p, p^a\}$ with $g$ a $q$-element (\emph{i.e.} $s=q$) and $t=p$. We distinguish two situations. Indeed, if $t=q$, then up to conjugation $xg=gx$, and by Lemma~\ref{basic} (c) both $|x^G|=q$ and $1\neq |g^G|$ divide $|(xg)^G|$; in particular, $|g^G|\in \{q, n\}$ due to Lemma~\ref{basic} (b) and the fact $|\pi(n)|\geq 3$. Moreover, if $|g^G|=q$, then $g$ must be a $p$-element: otherwise $yg=gy$ up to conjugation, and we would get $|(yg)^G|=pq=n$ which cannot occur. Similarly, if $t=p$ then either $|g^G|=n$, or $|g^G|\in\{p, p^a\}$ with $g$ a $q$-element.

Now let us consider a prime $r\in \pi(n)\smallsetminus \{p,q\}$. We can take a Sylow $r$-subgroup $R$ of $G$ such that $R\leqslant\ce{G}{x}\cap\ce{G}{y}$, and some $a\in R\smallsetminus \ze{G}$. Hence $xa=ax$ and $ya=ay$ with $(o(a),o(x))=1=(o(a),o(y))$ since both $x$ and $y$ are $t$-elements with $t\in\{p,q\}$. Consequently, if $|a^G|\in\{q,p,p^a\}$, then it follows that $|\pi(n)|=2$ which cannot occur. So $|a^G|=n$ and Lemma~\ref{ito} yields $\ce{G}{a}=R_0\times A$ with $R_0$ a Sylow $r$-subgroup of $\ce{G}{a}$ and $A$ an abelian $r'$-group. In particular $x,y\in A$. In virtue of the claim proved in the last paragraph, if $|g^G|\neq n$, then up to conjugation $a\in \ce{G}{g}$, so $g\in A$ and thus $g$ commutes with both $x$ and $y$. Since either $|g^G|=q$ with $g$ a $p$-element and $t=q$, or $|g^G|\in\{p, p^a\}$ with $g$ a $q$-element and $t=p$, we can appropriately mix $g$ with either $x$ or $y$ to produce a conjugacy class in $G$ with size a $\{p,q\}$-number, which must be equal to $n$, but this is not possible. 

From the last two paragraph we conclude that any prime power order $t'$-element of $G$ has class size either $1$ or $n$. An application of Theorem~\ref{theo-2-reg} produces the final contradiction $|\pi(n)|\leq 2$.
\end{proof}

\medskip

\begin{example}
For groups satisfying Theorem \ref{theoC} (a), \emph{i.e.} satisfying $cs(G)=\{1,p,p^a,n\}$, it may hold $|\pi(n)|> 2$, which differs from statements (b) and (c). For instance, let $G$ be the 166th group of order 480 in the library \texttt{SmallGroups} of GAP (\cite{GAP}), whose  structure is $C_3 \rtimes (C_5 \rtimes ((C_2 \times C_2 \times C_2) \rtimes C_4))$. Then one can check that $cs(G)=\{1,2,4,60\}$, so $n=60$ has three prime divisors. 
\end{example}

\begin{example}
In the previous example $G=\fitdos{G}$, but this property does not hold in general in Theorem \ref{theoC}; in other words, the other possibility for $G/\fitdos{G}$ claimed in Theorem \ref{theoC} may occur. For example, the group $G$ that is isomorphic to $((C_2 \times C_2 \times C_2 \times C_2) \rtimes C_5) \rtimes C_2$, which is the 234th group of order 160 in GAP, verifies $cs(G)=\{1,5,20,32\}$, so $G$ satisfies statement (b) of Theorem \ref{theoC}, and $G/\fitdos{G}$ is isomorphic to $C_2$.
\end{example}

\begin{example}
In Theorem \ref{theoC} (c), we have proved that if $cs(G)=\{1,p,q,p^a,n\}$ for some primes $q\neq p$, some integer $a\geq 2$, and some composite number $n$, then $n$ is a $\{p,q\}$-number. We strongly suspect that $n$ should indeed be equal to $qp^b$ for some integer $1\leq b\leq a$, as it happens in statement (b). An example of this fact is provided by the 176th group of order 486 in GAP, whose structure is $((C_3 \times ((C_3 \times C_3) \rtimes C_3)) \rtimes C_3) \rtimes C_2$. Here it holds $cs(G)=\{1,2, 3, 18, 27\}$, so $q=2$, $p^a=27=3^3$ and $n=18=qp^2$. Moreover, the 5th group of order 162 in GAP, whose structure is $((C_9 \times C_3) \rtimes C_3) \rtimes C_2$, yields another example of the aforementioned property: in this case $cs(G)=\{1,2,3,6,27\}$, so $q=2$, $p^a=27=3^3$ and $n=6=qp$.
\end{example}


\bigskip

\noindent \textbf{Declarations.} The authors declare no conflict of interest.


\end{document}